\documentclass[11pt]{article}

\usepackage{amsmath,amssymb,amsthm,mathrsfs}
\usepackage{fullpage}
\usepackage{hyperref}
\usepackage{enumitem}
\usepackage{color}

\title{Central limit theorems for non-linear functionals of Gaussian fields via Wiener chaos decomposition}
\author{Fabio Coppini\thanks{Mathematical Institute, Utrecht University, Budapestlaan 6, 3584 CD Utrecht, The Netherlands} \and Wioletta M. Ruszel\footnotemark[1]}
\date{}

\newtheorem{theorem}{Theorem}

\newtheorem{remark}{Remark}
\newtheorem{lemma}{Lemma}
\newtheorem{proposition}{Proposition}

\newcommand{\Z}{\mathbb{Z}}
\newcommand{\R}{\mathbb{R}}
\newcommand{\E}{\mathbb{E}}
\newcommand{\Cov}{\operatorname{Cov}}
\newcommand{\Var}{\operatorname{Var}}

\begin{document}

\maketitle

\begin{abstract}
    We review and present some known results for non-linear functionals of Gaussian variables in the context of discrete Gaussian fields defined on the $d$ dimensional lattice. Our main result is a Central Limit Theorem in the spirit of the classical Breuer-Major theorem, together with applications to the powers of the Gaussian Free Field. Notably, we show that even powers of the discrete Gaussian Free Field converge to the Gaussian white noise, while odd powers converge to a continuous Gaussian Free Field with explicit covariance. The proofs are based on the Wiener chaos decomposition and the fourth moment Theorem (Nualart-Peccati, 2005), and include a tightness result. Even if these tools are well-known in the literature, their application to Gaussian fields on the lattice appears to be new.
\end{abstract}

\section{Introduction}

\subsection{Motivation and main result}\label{subsec:intro}

Let $d\geq 2$, for $i\in \Z^d$ and $N=2, 3, \dots$, a natural number. Define $B_N(i) = \{j \in \Z^d : \vert i-j\vert \leq N/2\}$ a box of size $N$, centered at $i\in \Z^d$. Let $B_N$ be the shorthand notation for $B_N(o)$ with $o$ the origin in $\Z^d$. Furthermore, let $X = (X_j)_{j\in \Z^d}$ be a stationary discrete Gaussian Field on $\Z^d$ such that $\E[X_o]=0$ and $\E[X^2_o]=1$. For every $u\in \Z^d$, denote the covariance $\rho(u):=\E[X_o X_u]$.
Let $H:\R \to \R$ be a real valued function. For every $N$, define the field $\Phi_N$ as
\begin{equation*}
\Phi_N(j) := H(X_j), \quad j \in B_N.
\end{equation*}
Let $D=[-1/2, 1/2]^d$. The action of $\Phi_N$ on a test function $f:D \to \R$ is defined by
\begin{equation}
\label{def:langle Phi_N}
\langle \Phi_N, f \rangle := N^{-d} \sum_{j \in B_N} H(X_j) f\left(\frac{j}{N}\right).
\end{equation}
If $\Var[H(X_o)]<\infty$, then we have the following decomposition
\begin{equation*}
H(x) = h_0 + \sum_{q=m}^\infty c_q H_q(x), \quad c_i\in\R, \quad x\in \R,
\end{equation*}
into {\it Hermite polynomials} $(H_q)_{q\in \mathbb{N}}$ of order $q\geq m$, where $m\geq1$ is the \emph{Hermite rank} of $H$, i.e., the smallest $q \geq 1$ such that $c_q \neq 0$. Without loss of generality, we suppose that the field is centered, i.e., $\mathbb{E}[\langle \Phi_N, f \rangle]=0$ for every test function $f$ and $N\in\mathbb{N}$.

 Let $\{\varphi_k\}_{k=1}^\infty$ denote an orthonormal basis of $L^2(D)$ formed by $L^2$-normalized eigenfunctions of the Laplacian operator $-\Delta$ on $D$ with zero boundary condition, and let $\{\lambda_k\}_{k=1}^\infty$ be the corresponding eigenvalues. For any $\alpha>0$, the Sobolev space $\mathcal{H}^{-\alpha}(D) = (\mathcal{H}^{\alpha}(D))'$ is the dual of 
\begin{equation}
\label{def:hilbert-space}
  \mathcal{H}^\alpha(D)=\bigl\{ u=\sum_{k\ge1} u_k \varphi_k \text{ s.t. } \|(Id-\Delta)^{\alpha/2}u\|_{L^2}^2 := \sum_{k\ge1} (1+\lambda_k)^{\alpha} |u_k|^2 <\infty \bigr\}.
\end{equation}
The $d$-dimensional Gaussian {\it white noise} $W$ is the centered Gaussian random distribution defined for $f,g \in \mathcal{H}^{\alpha}(D)$,
\begin{equation*}
\mathbb{E}(\langle W, f\rangle \langle W,g\rangle) = \int_{D}f(x) g(x) dx.
\end{equation*}

We are interested in the conditions on $\rho$ and $H$ such that $\Phi_N$ satisfies a Central Limit Theorem (CLT), namely our main result is:

\begin{theorem}
\label{thm:main}
Let $(X_i)_{i \in \Z^d}$ be a centered stationary Gaussian field with covariance $\rho(u) = \E[X_o X_u]$ for $u\in \mathbb{Z}^d$. Let $H: \R\to \R$ be a function with Hermite expansion $H = \sum_{q=m}^\infty c_q H_q$ and Hermite rank $m \geq 1$. Suppose that
\begin{equation}
\label{hyp:susceptibility}
\sum_{u \in \Z^d} |\rho(u)|^m < \infty.
\end{equation}
Let $\Phi_N$ be the field defined in Equation\eqref{def:langle Phi_N}, then
\[
N^{d/2} \Phi_N \overset{d} \longrightarrow \sqrt{C_m} \,W 
\]
as $N\rightarrow \infty$ in $\mathcal{H}^{-\alpha}(D)$ for any $\alpha>d/2$ where $C_m=\sum_{q = m}^\infty q!{c_q^2} \sum_{u \in \Z^d} \rho(u)^q  \in [0, \infty)$.
\end{theorem}

CLT results for non-linear functionals of a stationary Gaussian vector are usually known as Breuer-Major theorems, see, e.g., \cite{breuer_central_1983, chambers_central_1989, nourdin_breuermajor_2021, nourdin_berry-esseen_2019, nualart_continuous_2020}, and involve the notion of Hermite rank, see., e.g., \cite{peccati_wiener_2011}. 
The distributional CLT setting presented here is slightly more general than the original result \cite{breuer_central_1983}: for instance, by taking $f\equiv 1$, then the limit of Equation \eqref{def:langle Phi_N} after centering and rescaling, i.e., the limit of
\begin{equation*}
N^{d/2} \left( \langle \Phi_N, 1 \rangle - \mathbb{E}\left[\langle \Phi_N, 1 \rangle\right]\right) = N^{-d/2} \sum_{j \in B_N} H(X_j) - \mathbb{E}[H(X_j)]
\end{equation*}
is known to be a Gaussian variable with explicit covariance, whenever the field covariance $\rho$ is such that $\sum_{u \in \Z^d} |\rho(u)|^m < \infty$.
More precisely, the statement of the classical Breuer-Major theorem is the following one.

\begin{theorem}[{\cite{breuer_central_1983}}]
\label{thm:breuer-major}
Let $(X_i)_{i \in \Z^d}$ be a centered stationary Gaussian field with unit variance and covariance $\rho(u) = \E[X_o X_u]$. Let $H: \R^d\to \R$ be a function with Hermite expansion $H = \sum_{q=m}^\infty c_q H_q$ and Hermite rank $m \geq 1$. Suppose that Equation \eqref{hyp:susceptibility} holds. Then, the normalized sums
\begin{equation}
Z^N_i = N^{-d/2} \sum_{j \in  B_N(i)} H(X_j), \quad i \in \Z^d
\end{equation}
converge in distribution to a centered Gaussian random variable with variance $C_m$ as defined in Theorem \ref{thm:main}.
Moreover, the finite dimensional distributions of the field $(Z^N_i)_{i\in \Z^d}$ tend, as $N\to\infty$, to the finite dimensional distributions of the field $\sigma Z^*_i$, where $(Z^*_i)_{i\in\Z^d}$ are independent standard normal random variables.
\end{theorem}

In the last 20 years, several proofs and extensions of the above result have been proposed, see for instance, \cite{nourdin_steins_2009} and the lecture notes  \cite{nualart_malliavin_2019} with the references therein. Generalizations of Theorem \ref{thm:breuer-major} are now present in the field of c\`adl\`ag functions, \cite{arcones_limit_1994,hu_renormalized_2005, nourdin_functional_2020}, self-similar Gaussian processes \cite{campese_continuous_2020, nualart_continuous_2020}, together with precise quantitative bounds for the rate of convergence \cite{nourdin_breuermajor_2021}. The literature being too vast to be fully cited, we refer to Subsection \ref{ss:literature} for a general overview. The interest in Gaussian fields and corresponding fluctuations have consistently risen in the same period, see, e.g., \cite{peccati_gaussian_2005, newman_gaussian_2018, giacomin_equilibrium_2001, cipriani_properties_2023, wu_local_2025}. These two fields appear to have been developed rather independently, with few results using the Wiener chaos decomposition approach to study non-linear functionals of Gaussian fields. To the authors' knowledge, only two applications exist in this spirit, i.e., \cite{gass_spectral_2025, mcauley_limit_2025}, targeting local non-linear functionals of Gaussian fields resp. the Gaussian Free Field, such as nodal volumes or level-set clusters.

The aim of this work is twofold: a Breuer-Major theorem in the case of discrete Gaussian fields, which appears to be missing in the existing literature, and one application to the case where $X$ is the discrete Gaussian Free Field (GFF) and $H(x)=x^p$, showing that even powers  converge to the white noise, while odd powers do not, see Section \ref{ss:application} for more on this. For the first part, we do not look at convergence of $Z^N_i$, but we rather tackle the field $\Phi_N$, and show that under standard hypothesis it converges, as a distribution, to the Gaussian white noise. A major difference with the original result is that no combinatorial estimate is used. The original proof is built on top of the method of moments, i.e., the moments of a general vector $\sum_{j=1}^k b_j Z^N_{j}$, with $b_j \in \R$, are studied and computed using Feynman diagrams and the corresponding formula.
In the proof proposed here, we take advantage of the fourth moment Theorem, introduced in the celebrated paper \cite{nualart_central_2005}, where Nualart and Peccati characterize the Gaussianity of a Wiener-It\^o integral in terms of its fourth moment, or via a nullity condition on a sequence of carefully constructed Hilbert norms. We precisely rely on this last characterization and show that the finite dimensional distributions of the field $\Phi_N$ can be directly expressed as It\^o integrals with respect to an underlying Wiener process. Proving Breuer-Major type results with the fourth moment Theorem has been the standard approach in the field for over a decade, however their application to discrete lattice configuration appears to be new and mathematically interesting. Notably, our approach has been inspired by the lecture notes \cite[Theorem 3.1.1]{nualart_malliavin_2019} where a simplified version of Theorem \ref{thm:breuer-major} is proven. We refer to Proposition \ref{pro:clt-vector-one-q} for the main contribution. The proof of tightness given here, i.e., Lemma \ref{lem:tightness}, also appears to be missing in the existing literature.

\begin{remark}
    In principle one can obtain tightness in a more general class of distributional spaces, namely Besov spaces, as done in \cite{cipriani_properties_2023}. However, we refrain from giving a result in this generality as defining the Besov framework is rather technical.
\end{remark}

We conclude this section by observing that the statement of Theorem \ref{thm:main} generalises the original Breuer-Major result; it addresses the convergence of $\Phi_N$ as a field instead of the finite dimensional distributions of $(Z^N_i)_{i\in \Z^d}$. For instance, for any $k \geq 2$, consider $Q_1, \dots, Q_k \subset D$ disjoint boxes and 
\begin{equation}
B_N^{(i)} = \{j\in \Z^d: j/N \in Q_i \}, \quad i=1, \dots, k.
\end{equation}
Let $f^{(i)}$ be the function defined by
\begin{equation}
f^{(i)}(x) = |Q_i|^{-d/2} 1_{Q_i}(x), \quad x \in D.
\end{equation}
We have that
\begin{equation}
\sqrt{\vert Q_i \vert } \langle \Phi_N, f^{(i)}\rangle = N^{-d/2} \sum_{j\in B^{(i)}_N} H(X_j) = \bar{Z}^N_i.
\end{equation}
By suitably approximating the functions $f^{(i)}$ with a $L^2$ sequence, Theorem \ref{thm:main} yields that the vector $(\bar{Z}^N_1, \dots, \bar{Z}^N_k)$ converges to a vector of independent standard normal random variables. As the set $B_N^{(i)}$ is morally the same of $B_N(i)$ for some choice of $(Q_i)_{i=1, \dots, k}$, we can transport this result to the variables $(Z^i_N)_{i \in \Z^d}$ and obtain Theorem \ref{thm:breuer-major}.

\subsection{Related works}
\label{ss:literature}

Theorem \ref{thm:breuer-major} is a celebrated CLT result for Gaussian processes, see also \cite{giraitis_clt_1985, arcones_limit_1994}. The literature is vast and hard to fully cite, we review here the main results and generalizations. Important contributions concern quantitative bounds \cite{nourdin_breuermajor_2021, nourdin_berry-esseen_2019, Vidotto},  and application to time processes \cite{nualart_continuous_2020, nourdin_functional_2020, campese_continuous_2020}. Multidimensional version of Breuer-Major have been established for example in \cite{bardet_moment_2013}, see also the discussion and the references cited in \cite[\S 11]{peccati_wiener_2011}. The techniques used in these works concern the notion of Wiener chaos and apply tools like Stein’s method and Malliavin calculus, see for instance the seminal work \cite{nourdin_steins_2009} and references therein for more on this perspective.

Fields with finite susceptibility are known to converge to the white-noise provided stationarity and FKG-like inequalities, see the beautiful result of Newman \cite{newman_normal_1980} on this perspective. In the case of linear and non-linear functionals of Gaussian fields, we cite \cite{cipriani_properties_2023, coppini_wick_2025, gass_spectral_2025}. To some extent, grad-phi models can exhibit a behavior close to the one of the GFF and several CLT for non-linear functionals have also be studied in literature, e.g., \cite{newman_gaussian_2018, giacomin_equilibrium_2001, miller_fluctuations_2011, wu_local_2025}. When the susceptibility of the field is infinite, the Dobrushin-Major result \cite{dobrushin_non-central_1979} establishes a different limit (and normalization!) characterized by a multiple Wiener integral representation.

Looking at the fourth moment Theorem \cite{nualart_central_2005}, many generalizations have been proposed, see, e.g., \cite{azmoodeh_generalization_2016,naganuma_generalizations_2022,duker_fourth-moment_2025}, as well as applications to the study of local functionals of the GFF, but also to more general results in free probability and statistics \cite{bourguin_four_2019, dobler_fourth_2018, kemp_wigner_2012}. 

Finally, tightness has been thoroughly studied in \cite{nourdin_functional_2020, campese_continuous_2020} in the context of 1$d$ (discrete and continuous) time processes converging to the Brownian motion. In this context, a careful control on time increments via $L^p$ bounds with $p>2$ and Malliavin calculus is needed to obtain tightness in the Skorohod space $D([0,1])$. In our case, the limit being the white noise in a sufficiently negative Hilbert space, it is enough to directly bound the $L^2$ norm of the process via the spectral properties of the Laplacian and use the fundamental orthogonality of the Hermite expansion.

\subsection{Applications to the powers of the discrete Gaussian Free Field}
\label{ss:application}
Recently, there has been an increasing interest in studying powers of the GFF, see the works \cite{cipriani_properties_2023,coppini_wick_2025} and \cite{chiarini_fermionic_nodate} for the connection with Fermionic fields. We now present a useful application of Theorem \ref{thm:main}, when applied to even powers of the GFF and study the odd powers separately.

\subsubsection{Even powers}

In \cite{cipriani_properties_2023, coppini_wick_2025}, it is shown that the even powers of the gradient of the GFF converge to the white noise in an open set $U\subset \R^d$, with $d\geq 3$. Here, we observe that this is mainly a consequence of Theorem \ref{thm:main}. Indeed, the class of functions $H(x)=x^{2p}$ has Hermite rank $2$ for all $p=1,2,\dots$, and the discrete GFF (defined as in \cite[Definitions 5 and 6]{cipriani_properties_2023}) has covariance $\rho$ which behaves like $|u|^{-(d-2)}$.
Theorem \ref{thm:main} ensures that every even power of $(X_i)_{i\in\Z^d}$ converges to the Gaussian white noise as soon as $d\geq 5$, in fact
\begin{equation*}
\sum_{u\in\Z^d}\rho(u)^2 \lesssim
\sum_{u\in\Z^d}|u|^{4-2d} < \infty.
\end{equation*}
A similar argument goes through when the gradient of the field is considered, in this case $\rho(u)$ behaves like $|u|^{-d}$ and Theorem \ref{thm:main} holds for every $d\geq 3$.

\begin{remark}
The work \cite{cipriani_properties_2023} addresses a non-stationary GFF as they consider it defined on $U\subset\R^d$, with $U$ generic open set.  As discussed in the original paper \cite[Equation (1.4b)]{breuer_central_1983}, the stationarity assumption can be dropped in favor of the uniform condition: \newline $\sup_m \sum_n |\mathbb{E}[X_m X_{m+n}]|^k< C<\infty$. In the case of the GFF, this condition is known to hold. To the authors’ knowledge, the only other work addressing a CLT for a field defined on a generic domain is \cite{miller_fluctuations_2011}. It is also worth mentioning \cite{bardet_moment_2013} which extends the Breuer-Major theorem to non-stationary triangular arrays of Gaussian vectors.
\end{remark}

\subsubsection{Odd powers}

The function $H(x) = x^{2p+1}$ has Hermite rank $1$ for every $p\geq 1$, i.e., the first order term $c_1$ in the Hermite expansion does not cancel out. For instance, $x^3 = 3H_1(x) + H_3(x)$ and similarly for $x^{2p+1}$ for $p>1$. Let $G$ be the Green's function of the discrete Laplacian on $\Z^d$ for $d\geq 3$, solving $(-\Delta) G = \delta_o$. The hypothesis of the Breuer-Major theorem are not satisfied, in fact
\begin{equation*}
\sum_{j \in \Z^d} |\Cov(X_o, X_j)| = \sum_{j \in \Z^d} |G(o,j)| = \infty.
\end{equation*}
In this case the limit is a continuous Gaussian Free Field. In the sequel $G_\textup{cont}$ refers to the continuous Green function on $D$, defined via the spectral decomposition $G_{\textup{cont}}(x,y) = \sum_{k=1}^\infty \lambda_k^{-1} \varphi_k(x) \varphi_k(y)$, where $\{\varphi_k\}_{k=1}^\infty$ and $\{\lambda_k\}_{k=1}^\infty$ are the eigenfunctions and eigenvalues of the Dirichlet Laplacian on $D$ introduced in Equation \eqref{def:hilbert-space}.

\begin{theorem}
    Let $X = (X_i)_{i\in \Z^d}$ be a Gaussian Free Field in $\Z^d$ for $d\geq 3$. Let $H(x):=x^{2p+1}$ for $p\geq 0$. Let $\Phi_N$ be the field defined in Equation \eqref{def:langle Phi_N}. Then, for any $\alpha>d/2 -1$ the centered field $\frac{N^d/2-1}{\sqrt{C_p}}\Phi_N$ converges in law in $\mathcal{H}^{-\alpha}(D)$ to a continuous Gaussian Free Field $\Phi$ with covariance
    \begin{equation*}
    \mathbb{E}[\langle \Phi, f\rangle\langle\Phi,g\rangle]= \int_{[-1/2, 1/2]^{2d}} f(x)g(y)G_{\textup{cont}}(x,y) dx dy
    \end{equation*}
    where $C_p=G(o,o)^{p+1} (2p+1)!!$ is the first order term $c_1$ in the decomposition of $H = \sum_j c_j H_j$.
\end{theorem}

\begin{proof}
    We give a sketch of the proof as most of the tools are standard or used in a more general fashion for proving Theorem \ref{thm:main}. As $H$ is an odd polynomial, we can write
    \begin{equation*}
        H(X_i) = X^{2p+1}_i = c_1 X_i + R_i
    \end{equation*}
    where $c_1 = \E[X^{2p+1}_o X_o] = \E[X^{2p+2}_o] = G(o,o)^{p+1} (2p+1)!! > 0$ and $R_i$ is a finite sum of Hermite polynomials of degree greater or equal than 3. Compute
    \begin{equation*}
        \Var \left[ N^{-d/2+1} \sum_{i\in B_N} f(i/N)X^{2p+1}_i \right] = c^2_1 N^{-(d-2)} \sum_{i,j \in B_N} f(i/N)f(j/N) G(i,j) + N^{-d} \sum_{m\geq 2} \sigma_{m, N},
    \end{equation*}
    where $\sigma_{m,N}$ is the variance contributed by the $(2m+1)$-th projection. We now claim that as $N$ to infinity, (1) the first term converges to $c^2_1 \int \int f(x)f(y) G_{\textup{cont}}(x,y) dx dy$ while (2) the second term converges to zero. (1) is a standard Riemann sum approximation, similar to what is obtained in the first part of Proposition \ref{pro:clt-vector-one-q}. The proof of (2) is more involved, but follows the main ideas of the proof of Proposition \ref{pro:clt-vector-one-q}, i.e., equivalently said $H_q(X)$ converges to the Gaussian white noise for any $q=2m+1$ with $m\geq 1$ integer. Split $\Phi_N = L_N + R_N$, where for $f$ test function $L_N$ and $R_N$ are defined by
    \begin{equation*}
        \langle L_N, f\rangle = S_{N,1}(f), \qquad \langle R_N, f\rangle = \sum_{q\geq 3} S_{N,q}(f),
    \end{equation*}
    with $S_{N,q}(f) = N^{-d/2-1} c_q \sum_{j\in B_N} H_q(X_j) f(j/N)$ for every $q\geq 1$, see Equation \eqref{def:S_{N,q}} for more on this decomposition.
    On the one hand, the field $R_N$ is tight as shown in Lemma \ref{lem:tightness}; on the other hand, $L_N$ only involves Gaussian variables.  Notably, let $\mathcal{G}_N$ be the operator on the lattice corresponding to the scaled kernel $N^{d-2}G(i,j)$, then we have that $\mathbb{E}|\langle L_N, \varphi_k\rangle|^2=c^2_1 N^{-2d}\sum_{i,j\in B_N}\varphi_k(i/N)\varphi(j/N) \mathcal{G}_N(i,j)$. 
    As $N$ goes to infinity, $\mathbb{E}|\langle L_N, \varphi_k\rangle|^2$ is thus close to $ c^2_1 \langle \varphi_k, G_\textup{cont} \varphi_k \rangle_{L^2(D)}$, with $G_\textup{cont}$ the continuous Green function on $D$, and we can bound the continuous inner product by $C\lambda_k^{-1}$ for some constant $C>0$.
    In particular,
    \begin{equation*}
        \mathbb{E}\left[ \Vert L_N \Vert^2_{\mathcal{H}^{-\alpha}(D)}\right] = \sum_{k\in\Z^d} (1+\lambda_k)^{-\alpha} \mathbb{E}\left[\langle L_N, \varphi_k\rangle^2\right] \leq C\sum_{k\in\Z^d} \lambda_k^{-\alpha-1} 
    \end{equation*}
    which is finite for every $\alpha > d/2 -1$ because of Weyl's law for the Dirichlet Laplacian on a bounded domain, which imples that $\lambda_k \sim k^{2/d}$ for large $k$. The proof is concluded given the properties of the orthogonal decomposition $\Phi_N = L_N+R_N$.
\end{proof}

\section{Preliminaries}

\subsection{Isonormal Wiener processes representation}

Let $\mathcal{H}$ be the Hilbert space constructed as the closure of finite linear combinations of $(u_j)_{j\in\Z^d}$, with inner product defined by
\begin{equation*}
    \langle u_j, u_k \rangle = \rho(j-k), \qquad j,k\in\Z^d.
\end{equation*}
The existence (and not uniqueness) of such a class of Hilbert spaces is classical in the literature, see, e.g., \cite{nualart_central_2005}. Let $W$ be the isonormal Gaussian process on $\mathcal{H}$, i.e., the family $W = (W(u))_{u \in \mathcal{H}}$ of centered Gaussian variables indexed by the elements of $\mathcal{H}$ and defined on some probability space $\left(\Omega, \mathcal{F}, \mathbb{P}\right)$. The process $W$ is such that, for every $u,v\in \mathcal{H}$, $\mathbb{E}\left[W(u)W(v)\right] = \langle u, v \rangle$. In particular, for every $j\in \Z^d$, we have the representation $X_j = W(u_j)$ and it holds that
\begin{equation*}
    \mathbb{E}\left[X_j X_k\right] = \mathbb{E}\left[W(u_j)W(u_k)\right] = \langle u_j, u_k \rangle = \rho(j-k), \qquad j\in\Z^d.
\end{equation*}
Fix $q\geq 2$ and let $u_j^{\otimes q}$ be the algebraic tensor product of $u_j$ and $u_j^{\odot q}$ its symmetric tensor product which respectively live in $\mathcal{H}^{\otimes q}$ and $\mathcal{H}^{\odot q}$. We denote by $I_q: \mathcal{H}^{\odot q} \to L^2(\Omega)$ the isometry between $\mathcal{H}^{\odot q}$ equipped with norm $\sqrt{q}\Vert \cdot \Vert_{\mathcal{H}^{\otimes q}}$ and the $q$-th Wiener chaos of $W$, i.e., $I_q$ is such that
\begin{equation}
\begin{split}
    &I_q(u_j^{\otimes q}) = H_q(X_j) \hspace{5cm}\qquad j \in \Z^d,\\
    & \mathbb{E}\left[I_q(u_j^{\otimes q}) I_q(u_k^{\otimes q})\right] = q! \langle u_j^{\otimes q} u_k^{\otimes q} \rangle = q! \rho(j-k)^k, \qquad j,k\in \Z^d,
\end{split}
\end{equation}
\label{def:I_q}
where the first two equalities are by definition and the last one is a simple computation of $\langle u_j^{\otimes q}, u_k^{\otimes q} \rangle$.

\subsection{Central limit theorem for the $q$-th components}
\label{ss:clt-q}

We start this section by recalling the seminal result of Nualart and Peccati.

\begin{theorem}[{\cite{nualart_central_2005}}]
\label{thm:fourth-moment-criterion}
    Fix $q\geq 2$. For any sequence of elements $\{f_N:N\geq1\}$ such that $f_N\in\mathcal{H}^{\odot q}$ for every $N$, and
    \begin{equation*}
        \lim_{N\to\infty} q! \Vert f_N \Vert^2_{\mathcal{H}^{\otimes q}} = \lim_{N\to\infty} \mathbb{E} \left[I_q(f_N)^2\right]=1,
    \end{equation*}
    the following conditions are equivalent:
    \begin{enumerate}
        \item[(i)] $\lim_{N\to\infty} \mathbb{E}\left[I_q(f_N)^4\right]=3$;
        \item[(ii)] for every $p=1, 2, \dots, q-1$, $\lim_{N\to\infty} \Vert f^{\otimes p}_N \Vert^2_{\mathcal{H}^{2(q-p)}} = 0$;
        \item[(iii)] as $N$ tends to infinity, the sequence $\{ I_q(f_N): N\geq 1 \}$ converges in distribution to a standard Gaussian random variable.
    \end{enumerate}
\end{theorem}

For every $q\geq 1$ and $f\in \mathcal{H}^{\alpha}(D)$, let $s_{N,q}(f)$ be the element of $\mathcal{H}^{\otimes q}$ defined by
\begin{equation}
    \label{def:s_{N,q}}
    s_{N,q}(f) = N^{-d/2} c_q \sum_{j\in B_N} u_j^{\otimes q} f(j/N).
\end{equation}
Recall that $H_q$ is the $q$-th Hermite polynomial and that $I_q(u_j^{\otimes q}) = H_q(X_j)$ for every $j\in \Z^d$. As $I_q$ is a linear functional, we can define $S_{N,q}(f)=I_q(s_{N,q})$, i.e., 
\begin{equation}
    \label{def:S_{N,q}}
    S_{N,q}(f) = I_q(s_{N,q}) = N^{-d/2} c_q \sum_{j\in B_N} H_q(X_j) f(j/N).
\end{equation}
Thanks to the fact that $H = \sum_q H_q$, one can (formally) write $\langle \Phi_N, f \rangle = \sum_q S_{N,q}(f)$.
In order to show that this equality is true for random variables (see Lemma \ref{lem:decomposition}), we start by giving a preliminary CLT for the $q$-th component $S_{N,q}(f)$.

\begin{proposition}
\label{pro:clt-vector-one-q}
    Assume hypothesis \eqref{hyp:susceptibility} and let $f\in \mathcal{H}^{\alpha}(D)$. For every $q\geq m$, it holds that
    \begin{equation*}
        \Var(S_{N,q}(f)) = q! c^2_q N^{-d} \sum_{u\in \Z^d} \rho(u)^q \sum_{\substack{j,k\in B_N\\j-k=u}} f\left(\tfrac jN \right) f\left(\tfrac kN \right).
    \end{equation*}
    As $N$ tends to infinity, it holds that
    \begin{equation*}
        \langle S_{N,q}, f \rangle \longrightarrow \mathcal{N}(0,  \sigma^2_q\int_{D}f^2(x)dx),
    \end{equation*}
    where
    \begin{equation}
    \label{def:sigma^2_q}
        \sigma^2_q = q! c^2_q \Big(\sum_{u\in \Z^d} \rho(u)^q \Big) < \infty.
    \end{equation}
\end{proposition}
\begin{proof}
From the definition of $I_q$, recall Equation \eqref{def:I_q}, we already know that $\E\left[ H_q(X_j)H_q(X_k)\right] = q! \, \rho(j-k)^q$.
A straightforward computation of the variance yields
\begin{equation*}
    \Var(S_{N,q}) = c^2_q N^{-d} \sum_{j,k\in B_N} q!  \rho(j-k)^q f\left(\tfrac jN \right) f\left(\tfrac kN \right).
\end{equation*}
The final expression for $\Var(S_{N,q})$ is obtained with the change of variable $u=j-k$. Observe that, for $f\equiv 1$ and assuming \eqref{hyp:susceptibility}, the last expression implies that $\sum_{j,k\in B_N} \rho(j-k)^q > 0$ for every $N>0$ large enough. We now claim that
\begin{equation*}
    \lim_{N\to\infty} \Var(S_{N,q}) = q! c^2_q \left(\sum_{u\in \Z^d} \rho(u)^q \right) \cdot \int_{D}f^2(x)dx.
\end{equation*}
This is a consequence of the fact that the sum $\sum_{\substack{j,k\in B_N\\j-k=u}} f\left(\tfrac jN \right) f\left(\tfrac kN \right)$ can be rewritten as
\[
\sum_{k\in B_N\cap(B_N-u)} f\left(\tfrac {k+u}N \right) f\left(\tfrac kN \right).
\]
Since $f$ is bounded and continuous and, for each fixed $u$, $u/N\to0$ as $N$ tends to infinity: we recognise the Riemann sum
\begin{equation*}
    \lim_{N\to \infty} N^{-d}\sum_{k\in B_N\cap(B_N-u)} f\left(\tfrac {k+u}N \right) f\left(\tfrac kN \right) = \int_{D} f^2(x) dx.
\end{equation*}
The claim is proved.

In order to prove the CLT for $S_{N,q}$ and thus conclude the proof, we use the $(ii)$ characterization of Theorem \ref{thm:fourth-moment-criterion} with $f_N = s_{N,q}$. We have to prove that, for every $1\leq r\leq q-1$, it holds that
\begin{equation*}
    \lim_{N\to \infty} \Vert s_{N,q} \otimes_r s_{N,q}\Vert_{\mathcal{H}^{\otimes (2q-2r)}}=0.
\end{equation*}
Fix $1\le r\le q-1$. Using definition \eqref{def:s_{N,q}} and bilinearity of the contraction, a standard computation gives
\begin{equation}
\label{eq:contraction-explicit}
\begin{split}
    s_{N,q}\otimes_r s_{N,q} = c_q^2 N^{-d} \sum_{j_1,j_2\in B_N} f\left(\tfrac{j_1}{N}\right)f\left(\tfrac{j_2}{N}\right)
\rho(j_1-j_2)^r h_{j_1}^{\otimes(q-r)}\otimes h_{j_2}^{\otimes(q-r)}.
\end{split}
\end{equation}
Hence the squared norm in $\mathcal H^{\otimes(2q-2r)}$ is
\begin{equation}
\label{eq:contraction-norm-sq}
\begin{split}
&\big|s_{N,q}\otimes_r s_{N,q}\big|^2 
 \Big( \prod_{\ell=1}^4 f\left(\tfrac{j_\ell}{N}\right)\Big)
\rho(j_1-j_2)^r \rho(j_3-j_4)^r \rho(j_1-j_3)^{q-r}\rho(j_2-j_4)^{q-r}\\
& \leq C N^{-2d} \sum_{j_1,j_2,j_3,j_4\in B_N}
\rho(j_1-j_2)^r \rho(j_3-j_4)^r
\rho(j_1-j_3)^{q-r} \rho(j_2-j_4)^{q-r}
\end{split}
\end{equation}
for some constant $C=C(q,c_q, \sup_{x\in[-1/2,1/2]^d}|f(x)|)$, recall that $f$ is bounded. Apply the elementary inequality (valid for nonnegative $a,b$ and integers $0\leq r\leq q$), $a^rb^{q-r} \leq a^q + b^q$ to the pair $a=\rho(j_1-j_2)$, $b=\rho(j_1-j_3)$. Then
\begin{equation*}
    \rho(j_1-j_2)^r 
\rho(j_1-j_3)^{q-r} \leq \rho(j_1-j_2)^q +
\rho(j_1-j_3)^q.
\end{equation*}
Multiplying by the remaining factors $\rho(j_3-j_4)^r 
\rho(j_2-j_4)^{q-r}$ and summing, produce the bound
\begin{equation*}
\begin{split}
    \big|s_{N,q}\otimes_r s_{N,q}\big|^2 &\leq C N^{-2d} \sum_{j_1, j_2, j_3, j_4\in  B_N} \left[\rho(j_1 -j_2)^q + \rho(j_1-j_3)^q\right]\rho(j_3-j_4)^r \rho(j_2-j_4)^{q-r}\\
    & \leq 2CN^{-2d} \sum_{j_1, j_2, j_3, j_4\in  B_N} \rho(j_1 -j_2)^q\rho(j_3-j_4)^r \rho(j_2-j_4)^{q-r},
\end{split}
\end{equation*}
where the second inequality comes from symmetric considerations. Fix the variables
\begin{equation*}
    u:=j_1-j_2, \quad v:=j_3-j_4, \quad w:=j_2-j_4.
\end{equation*}
For a given $(u,v,w)$ and a choice $j_4\in B_N$, the quadruple $(j_1, j_2, j_3, j_4)$ is determined. Hence, for any fixed $(u,v,w)$, the number of quadruples $(j_1, j_2, j_3, j_4)$ is at most $|B_N|$. Therefore, we can write
\begin{equation*}
\begin{split}
    \big|s_{N,q}\otimes_r s_{N,q}\big|^2 &\leq C N^{-d}  \sum_{u \in  \Z^d} \rho(u)^q \Big(\sum_{\substack{v\in \Z^d\\ |v|\leq N}}\rho(v)^r \Big)\Big(\sum_{\substack{w\in \Z^d\\ |w|\leq N}}\rho(w)^{q-r} \Big)\\
    &= C \Big(\sum_{u \in  \Z^d} \rho(u)^q \Big)\Big(N^{-d(1-\frac rq)}\sum_{\substack{v\in \Z^d\\ |v|\leq N}}\rho(v)^r \Big)\Big(N^{-d(1-\frac {q-r}q)}\sum_{\substack{w\in \Z^d\\ |w|\leq N}}\rho(w)^{q-r} \Big).
\end{split}
\end{equation*}
Focus on the term with $\rho(v)^r$ (the term with $q-r$ is equivalent). It remains to show that for $r=1, \dots, q-1$, we have that
\begin{equation*}
    \lim_{N\to\infty}N^{-d(1-\frac rq)}\sum_{\substack{v\in \Z^d\\ |v|\leq N}}\rho(v)^r =0.
\end{equation*}
Fix $\delta>0$ and split the sum over the sets $\{0\leq |v|\leq \delta N\}$ and $\{\delta N < |v| \leq N \}$. For the first contribution, we apply H\"older's inequality and obtain that
\begin{equation*}
    N^{-d(1-\frac rq)}\sum_{\substack{v\in \Z^d\\ |v|\leq \delta N}}\rho(v)^r \leq C \delta^{d(1-\frac rq)} \Big(\sum_{v\in \Z^d}\rho(v)^q\Big)^{r/q} \leq C' \delta^{d(1-\frac rq)},
\end{equation*}
which goes to zero as $\delta$ tends to zero. Similarly, for the second term we obtain
\begin{equation*}
    N^{-d(1-\frac rq)}\sum_{\substack{v\in \Z^d\\ \delta N \leq |v|\leq  N}}\rho(v)^r \leq C \Bigg(\sum_{\substack{v\in \Z^d\\ \delta N \leq |v|\leq  N}}\rho(v)^q\Bigg)^{r/q}
\end{equation*}
which goes to zero as $N$ tends to infinity because of the summability properties of $\rho$. The proof is concluded.
\end{proof}

It is straightforward to leverage the previous CLT to the convergence of Gaussian vectors of following type.
\begin{lemma}
\label{lem:clt-vector-different-q}
    For every $q\geq 1$, the vector $\left(S_{N,1}(f), S_{N,2}(f)\dots, S_{N,q}(f)\right)$ converges in law to the vector $\left(S_1, \dots, S_q\right)$, where $\{S_q, q\geq 1\}$ are independent centered Gaussian random variables with variance $\{\sigma^2_q, q\geq 1\}$  with $\sigma^2_q$ defined in Equation \eqref{def:sigma^2_q}.
\end{lemma}
\begin{proof}
    See, e.g., \cite[Theorem 1]{peccati_gaussian_2005}.
\end{proof}

\section{Proof of Theorem \ref{thm:main}}

\subsection{Convergence of observables}

\begin{lemma}
\label{lem:decomposition}
For $q\geq m$, the quantities $S_{N,q}=N^{-d/2} c_q \sum_{j\in B_N} H_q(X_j) f(j/N)$ defined in Equation \eqref{def:S_{N,q}} satisfy
\begin{equation*}
    \langle \Phi_N, f \rangle = \sum_{q=m}^\infty S_{N,q}(f).
\end{equation*}
As $N$ tends to infinity
\begin{equation*}
    \langle \Phi_N, f \rangle \to \mathcal{N}\left(0, \sum_{q=m}^\infty \sigma^2_q \int_D f^2(x) d x \right).
\end{equation*}
\end{lemma}
\begin{proof}
This result is standard and, to the authors' knowledge, appeared for the first time in \cite{hu_renormalized_2005}, see also the discussion in \cite[Theorem 11.8.3]{peccati_wiener_2011}.
We also observe that, $C_m$ defined in Theorem \ref{thm:main} is equal to $\sum_{q\geq m}\sigma^2_q$ with $\sigma_q$ defined in Equation \eqref{def:sigma^2_q}. We defer the proof of $C_m <\infty$ to Lemma \ref{lem:tightness}.
\end{proof}

\begin{lemma}
    For any finite collection of test functions $f_1,\dots,f_k\in \mathcal{H}^{\alpha}(D)$, it holds that 
    \begin{equation*}
    \bigl(\langle\Phi_N,f_i\rangle\bigr)_{i=1}^k 
        \;\xrightarrow[N\to\infty]{d}\;
        \mathcal N\bigl(0, \Sigma\bigr),
    \end{equation*}
    where the covariance matrix $\Sigma = (\Sigma_{ij})_{1\leq i,j\leq k}$ is given by
    \begin{equation*}
        \Sigma_{ij} = \sum_{q\geq m} \sigma^2_q \langle f_i,f_j\rangle_{L^2(\R^d)}.
    \end{equation*}
\end{lemma}

\begin{proof}
It suffices to test convergence of all linear combinations. Fix arbitrary coefficients $\alpha_1, \dots, \alpha_k\in \R$ and set
\begin{equation*}
    f=\sum_{i=1}^k \alpha_i f_i, \quad \langle \Phi_N, f \rangle = \sum_{i=1}^k \alpha_i \langle \Phi_N, f_i \rangle.
\end{equation*}
From Lemma \ref{lem:decomposition}, we have the decomposition $\langle \Phi_N, f \rangle= \sum_{q=m}^\infty S_{N,q}(f)$. For each $q\geq m$, $S_{N,q}(f)$ converges in distribution to a centered Gaussian variable with variance $\sigma^2_q \Vert f\Vert^2_{L^2(D)}$, recall Proposition \ref{pro:clt-vector-one-q}. Moreover, by Lemma \ref{lem:clt-vector-different-q}, the limits $(S_{N,q}(f))_q$ are asymptotically independent for different $q$. The proof is concluded by observing that $\sum_{q\geq m} \sigma^2_q(\sum_{i=1}^k \alpha_i f_i) = \alpha^{\intercal} \Sigma \alpha$.
\end{proof}

\subsection{Tightness}
The proof of Theorem \ref{thm:main} is concluded modulo the last lemma.

\begin{lemma}
\label{lem:tightness}
    For any $\alpha>d/2$, the sequence $\Phi_N$ satisfies
    \begin{equation*}
    \sup_N \E\left[ \bigl\Vert\Phi_N\Vert_{\mathcal{H}^{-\alpha}}^2 \right] <\infty.
    \end{equation*}
    In particular, $\Phi_N$ is tight in $\mathcal{H}^{-\alpha}(D)$ and $C_m=\sum_{q=m}^\infty q! c^2_q \sum_{u\in\Z^d} \rho(u)^q$ defined in Theorem \ref{thm:main} is finite.
\end{lemma}

\begin{proof}
Let $\{\varphi_k\}_{k=1}^\infty$ denote an orthonormal basis defined in Subsection \ref{subsec:intro}.
Recall that the spectrum of $\Delta$ is discrete, $0\le \lambda_1\le\lambda_2\le\cdots$, and that $\lim_{k\to\infty}\lambda_k/k^{2/d}= c_d$ because of Weyl's asymptotics. 
Using $\alpha>d/2$ and the fact that $D$ is bounded, we have in particular that,
\begin{equation}
\label{eq:laplacian-property}
\sup_{x,y\in D} \sum_{k\ge1} (1+\lambda_k)^{-\alpha} \varphi_k\!\bigl(x)\varphi_k\!\bigl(y) < \infty.
\end{equation}
The squared norm of $\Phi_N$ is given by 
\begin{equation}
\label{eq:Hs-norm-sum}
  \|\Phi_N\|_{\mathcal{H}^{-\alpha}}^2
  \;=\;\sum_{k\ge1} (1+\lambda_k)^{-\alpha} \bigl|\langle\Phi_N,\varphi_k\rangle\bigr|^2,
\end{equation}
with $\langle\Phi_N,\varphi_k\rangle
  = N^{-d/2}\sum_{j\in B_N} H(X_j)\,\varphi_k\!\bigl(\tfrac jN\bigr)$.
Using the Hermite expansion $H=\sum_{q=m}^\infty c_q H_q$ and the orthogonality properties of $H_q$, we obtain
\begin{align}
  \mathbb E\bigl|\langle\Phi_N,\varphi_k\rangle\bigr|^2
  &= N^{-d}\sum_{j,\ell\in B_N} \varphi_k\!\bigl(\tfrac jN\bigr)\varphi_k\!\bigl(\tfrac \ell N\bigr)
      \,\mathbb E\bigl[H(X_j)H(X_\ell)\bigr] \notag\\
  &= N^{-d}\sum_{j,\ell\in B_N} \varphi_k\!\bigl(\tfrac jN\bigr)\varphi_k\!\bigl(\tfrac \ell N\bigr)
     \sum_{q=m}^\infty q!\,c_q^{\,2}\,\rho(j-\ell)^q. \label{eq:coeff-var}
\end{align}
We now insert \eqref{eq:coeff-var} into \eqref{eq:Hs-norm-sum} and take the expected value:
\begin{equation*}
\begin{split}
  \mathbb E\bigl[\|\Phi_N\|_{H^{-\alpha}}^2\bigr]
  &= \sum_{k\ge1} (1+\lambda_k)^{-\alpha}
       \mathbb E\bigl|\langle\Phi_N,\varphi_k\rangle\bigr|^2 \notag\\
  &= N^{-d}\sum_{j,\ell\in B_N} \Bigl(\sum_{k\ge1} (1+\lambda_k)^{-\alpha} \varphi_k\!\bigl(\tfrac jN\bigr)\varphi_k\!\bigl(\tfrac \ell N\bigr)\Bigr)
       \sum_{q=m}^\infty q!\,c_q^{\,2}\,\rho(j-\ell)^q\\
 &\leq C N^{-d}\sum_{j,\ell\in B_N} \sum_{q=m}^\infty q!\,c_q^{\,2}\,\rho(j-\ell)^q\\
 &\leq C  \sum_{q=m}^\infty q!\,c_q^{\,2} \sum_{j\in B_N}\rho(j)^q <\infty,
\end{split}
\end{equation*}
where in the third step we have used Equation \eqref{eq:laplacian-property}. In particular, $C_m<\infty$. The proof is concluded with a standard application of Markov's inequality.
\end{proof}

\section*{Funding}
F.C. was supported by the NWO (Dutch Research Organization) grant OCENW.KLEIN.083 and W.M.R. by the NWO (Dutch Research Organization) grants OCENW.KLEIN.083 and VI.Vidi.213.112.

\subsection*{Acknowledgments}

The authors would like to thank G. Peccati for pointing out the reference \cite{mcauley_limit_2025} which inspired this work and the anonymous referees for their valuable feedback on the first version of this work.

\bibliographystyle{alpha}
\bibliography{biblio}

\end{document}